\documentclass[12pt]{article}
\usepackage{amsmath,amssymb}
\usepackage[T1]{fontenc}
\usepackage{latexsym}
\usepackage{pdfsync}
\usepackage{epsfig}
\usepackage{proof}
\usepackage[utf8]{inputenc}
\usepackage[a4paper, top=3cm, bottom=3cm ]{geometry}

\newtheorem{thm}{Theorem}[section]
\newtheorem{defn}{Definition}[section]
\newtheorem{lem}{Lemma}[section]

\newtheorem{rem}{Remark}[section]

\newtheorem{example}{Example}[section]

\title
{$L^p$-norm estimate for the Bergman projection on  Hartogs triangle}

\author{\normalsize Tomasz Beberok \\
\small Faculty of Mathematics and Computer Science, Jagiellonian University,\\
\small Lojasiewicza 6, 30-048 Krakow, Poland \\}

\date{}

\begin{document}

\begin{center}
  \textbf{Markov's inequality and $C^{\infty}$ functions on algebraic sets
}
\end{center}
\vskip1em
\begin{center}
  Tomasz Beberok
\end{center}

\vskip2em

\noindent \textbf{Abstract.} It is known that for $C^{\infty}$ determining sets Markov's property is equivalent to Bernstein's property. The purpose of this paper is to prove an analogous result in the case of compact subsets of algebraic varieties.
\vskip1em


\section{Introduction}
\label{}
Jackson's famous estimate of the error of the best polynomial approximation for a fixed function is one of the main theorems in constructive function theory. According to a multivariate version of the classical Jackson theorem (see e.g. \cite{Tim}), if $I$ is a compact cube in  $\mathbb{R}^N$ and $f :  I \rightarrow \mathbb{R}$ is a $\mathcal{C}^{k+1}$ function on $I$ then
   \begin{align*}
     n^k{\rm dist}_{I}(f,\mathcal{P}_n)  \leq  C_k \sum_{j=1}^{N} \sup_{x \in I} \left| \frac{\partial^{k+1} f }{\partial x_j^{k+1}} (x) \right|,
   \end{align*}
where the constant $C_k$ depends only on $N$, $I$ and $k$. As usual,
$ {\rm dist}_{I}(f,\mathcal{P}_n) =\inf\{\|f-p\|_{I} : p\in \mathcal{P}_n\}$, \ $\mathcal{P}_n$ is the space of all algebraic polynomials of degree at most $n$ and $\|\cdot\|_{I}$ is the sup norm on $I$. \newline
\indent As an application of Jackson's theorem, one can prove classical results like the well known Bernstein theorem (see e.g. \cite{DL}, \cite{Ch}) which allows to obtain a characterization of $C^\infty$ functions:

\vskip 2mm
\noindent {\it A function $f$ defined on $I$ can be extended to a $C^\infty$ function on $\mathbb{R}^N$ if and only if }
    \begin{align*} \lim_{n \rightarrow\infty} n^k  {\rm dist}_{I} (f,\mathcal{P}_n) \ = \ 0 \ \ \ \ \ {for \ all \ positive \ integer \ numbers \ k}.\end{align*}
A natural question arises: for which compact subsets $E$ of $\mathbb{R}^N$ the following Bernstein property holds?
\vskip 2mm
\noindent  {\it  for every function $f:E\to \mathbb{R}$ if the sequence $\{\text{dist}_E(f, \mathcal{P}_n)\}_n$ is rapidly decreasing (i.e. $\lim\limits_{n\to \infty} n^k {\rm dist}_E(f, \mathcal{P}_n) =0$ for all $k>0$), \ then there exists a $C^\infty$ function $F : \mathbb{R}^N \to \mathbb{R}$ such that $F=f$ on $E$.}
\vskip 2mm
In 1990, Pleśniak proved an important theorem \cite{P1} (see  also \cite{PP} for previous results) that answers to an above question and provides equivalence of the Markov inequality
   \begin{align*}
     \|D^{\alpha}P\|_E \leq M (\deg P)^{r|\alpha|} \|P \|_E, \quad  \alpha \in \mathbb{Z}^N_{+},
   \end{align*}
and Bernstein's property for $C^{\infty}$ determining sets. Our goal is to find a generalization of this fact for sets which are not $C^{\infty}$ determining. \newline
\indent The problems that appear naturally in various fields of science, especially in natural and technical science concern not only $C^\infty$ determining sets but also curves and varieties in spaces of several variables. Therefore, the specialists in approximation theory start to work with algebraic sets, analytic varieties etc. In the early seventies of the twentieth century, approximation problems on compact subsets of algebraic sets were considered in the pioneering Ragozin papers (\cite{Ragozin70}, \cite{Ragozin71}). Fundamental criterion of algebraicity for a pure-dimensional analytic subvariety in $\mathbb{C}^N$ is due to Sadullaev \cite{S}. This criterion indicates why polynomial approximation makes sense on compact subsets of algebraic varieties. Zeriahi (see \cite{Ze3,Ze2}) applied Sadullaev's criterion to obtain Bernstein-Walsh-Siciak theorems for compact subsets of algebraic varieties. \newline
\indent Some generalizations of the classical polynomial inequalities considered for curves and submanifolds in $\mathbb{R}^N$ led to show interesting characterizations. For instance, according to \cite{Bos95}, a $C^{\infty}$ submanifold $K$ of  $\mathbb{R}^N$ admits the tangential Markov inequality with the exponent one if and only if $K$ is algebraic. \newline
\indent Over the last few years, results concerning interpolation, polynomial inequalities and pluripotential theory on algebraic sets and analytic varieties have gained in interest, see e.g. papers by Ma'u (see e.g. \cite{BM}, \cite{BoM}), Bos, Brudnyi (see e.g. \cite{2010}, \cite{B}), Bos, Levenberg, Waldron (see e.g. \cite{BLW}), Cox (see e.g. \cite{CM}), Yomdin (see e.g. \cite{HY}, \cite{FY}), Izzo (see \cite{Izzo}), Fefferman \cite{Fef96}, Baran and Ple\'sniak \cite{2000a,2000b}, Skiba \cite{Sk} etc. This is an active branch of mathematics, see for instance the recent paper \cite{LCK} by Bia{\l}as-Cie\.{z}, Calvi and  Kowalska regarding polynomial inequalities on certain algebraic hypersurfaces.
\section{Markov inequality}
Our intention in this section is to study an extension of the Markov inequality to compact subsets of algebraic set. We will consider sets of the form
      \begin{align}\label{V}
        V=\left\{(x_1,\ldots,x_N) \in \mathbb{R}^N : x_k^d=Q_0(y)+Q_1(y)x_k + \cdots + Q_{d-1}(y)x_k^{d-1}\right\},
      \end{align}
where $Q_i$ are polynomials for every $0\leq i \leq d-1$ and $y=(x_1,\ldots,x_{k-1},x_{k+1},\ldots,x_N) \in \mathbb{R}^{N-1}$.
Since every polynomial $P$ from $\mathcal{P}(x_1,\ldots,x_N)$ , on $V$,  coincides with some polynomial from $\mathcal{P}(y) \otimes \mathcal{P}_{d-1}(x_k)$ (see \cite{LCK}), we have that the ring of polynomials on $V$ is
      \begin{align}\label{P(V)}
        \mathcal{P}(V):=\left\{P_{|V}, P \in \mathcal{P}(x_1,\ldots,x_N) \right\}=\left\{P_{|V}, P \in \mathcal{P}(y) \otimes \mathcal{P}_{d-1}(x_k) \right\},
      \end{align}
where $\mathcal{P}(y) \otimes \mathcal{P}_{d-1}(x_k)$ denotes the subspace of $\mathcal{P}(x_1,\ldots,x_N)$ formed of all polynomials of the form $\sum_{i=0}^{d-1} G_i(y)x_k^i$ with $G_i \in \mathcal{P}(y)$.
A considerations in \cite{BK} and \cite{LCK}  suggest the following definition
\begin{defn}[Markov set and Markov inequality on \textbf{F}]
  Let $\textbf{F}$ be an infinite dimensional subspace of $\mathcal{P}(x_1,\ldots,x_N)$ such that $P \in \textbf{F}$ implies $D^{\alpha}P \in \textbf{F}$ for all $\alpha \in \mathbb{Z}^N_{+}$. A compact set $E \subset \mathbb{R}^N$  is said to be a $\textbf{F}$-Markov set if there exist $M, m > 0$ such that
      \begin{align}\label{Markov}
        \|D^{\alpha} P\|_E \leq M^{|\alpha|} (\deg P )^{m|\alpha|} \|P\|_E, \quad P \in \textbf{F}, \quad \alpha \in \mathbb{Z}^N_{+}.
      \end{align}
This inequality is called a $\textbf{F}$-Markov inequality for $E$.
\end{defn}
Note that the space $\mathcal{P}(y) \otimes \mathcal{P}_{d-1}(x_k)$ is obviously invariant by derivation and it suffices to check the property for $|\alpha| = 1$. Now we give an example to demonstrate that the above definition makes sense.
    \begin{example}\label{Ex1}
      Let $V=\{y^3=(1-x^2)y\} \subset \mathbb{R}^2$. The compact set $E =\{(x,y) \in V : x \in [-1,1]\} $ is a $\mathcal{P}(x) \otimes \mathcal{P}_{2}(y)$-Markov.
    \end{example}
      \begin{proof}
        Let $P \in \mathcal{P}(x) \otimes \mathcal{P}_{2}(y)$. Then $P(x,y)=G_0(x) + G_1(x)y + G_2(x)y^2$ for some $G_i \in \mathcal{P}(x)$ for $i=0,1,2$. Now
           \begin{align*}
             \left\|D^{(1,0)} P(x,y) \right\|_E \leq& \left\|G'_0(x)\right\|_E + \left\|G'_1(x)y + G'_2(x)y^2\right\|_E \\&=\left\|G'_0(x)\right\|_{[-1,1]} + \left\|G'_1(x)y + G'_2(x)y^2\right\|_{E'}
           \end{align*}
where $E'=\{(x,y) \in \mathbb{R}^2 : y^2=1-x^2\}$. Since $E'$ is symmetric we have
         \begin{align*}
             \left\|D^{(1,0)} P(x,y) \right\|_E  \leq \left\|G'_0(x)\right\|_{[-1,1]} + \left\|G'_1(x)\sqrt{1-x^2}\right\|_{[-1,1]} + \left\|G'_2(x)(1-x^2)\right\|_{[-1,1]}.
         \end{align*}
By the classical Markov inequality, Bernstein’s inequality and \cite{LBC}, respectively, we get
     \begin{align*}
        \left\|D^{(1,0)} P(x,y) \right\|_E \leq& (\deg G_0)^2\left\|G_0(x)\right\|_{[-1,1]} + \deg G_1 \left\|G_1(x)\right\|_{[-1,1]} \\&+ 2(2 + \deg G_2)^2 \left\|G_2(x)(1-x^2)\right\|_{[-1,1]}
     \end{align*}
The classical inequality of Schur yields the following
     \begin{align*}
        \left\|D^{(1,0)} P(x,y) \right\|_E \leq& (\deg G_0)^2\left\|G_0(x)\right\|_{[-1,1]} + (\deg G_1+1)^2\left\|G_1(x)\sqrt{1-x^2}\right\|_{[-1,1]} \\&+ 2(2+\deg G_2)^2\left\|G_2(x)(1-x^2)\right\|_{[-1,1]}
     \end{align*}
Therefore we see immediately that
    \begin{align*}
      \left\|D^{(1,0)} P(x,y) \right\|_E \leq 2(\deg P)^2 \left(\left\|G_0(x)\right\|_{[-1,1]} + \left\|G_1(x)y+ G_2(x)y^2\right\|_{E'}\right)
    \end{align*}
The fact that $\|G_0(x)\|_{[-1,1]} \leq \|P\|_E$, together with the triangle inequality, implies
     \begin{align*}
       \left\|D^{(1,0)} P(x,y) \right\|_E \leq 6(\deg P)^2\left\|P\right\|_E.
     \end{align*}
Next, we consider the case of $D^{(0,1)}$. It is clear that
    \begin{align*}
      \left\|D^{(0,1)} P(x,y) \right\|_E \leq \left\|G_1(x)\right\|_E + \left\|G_2(x)y\right\|_E
    \end{align*}
Again by Schur's inequality and Lemma 2.4 in \cite{MB}, we have
     \begin{align*}
       \left\|D^{(0,1)} P(x,y) \right\|_E \leq&  (\deg G_1+1)\left\|G_1(x)\sqrt{1-x^2}\right\|_{[-1,1]} \\&+ (1+\deg G_2)^2\left\|G_2(x)(1-x^2)\right\|_{[-1,1]}
     \end{align*}
Now a similar proof to that of the previous case gives the following
     \begin{align*}
       \left\|D^{(0,1)} P(x,y) \right\|_E \leq 2(\deg P)^2\left\|P\right\|_E.
     \end{align*}
That is what we wished to prove.
\end{proof}
Next example shows that $\textbf{F}$-Markov inequality depends not only on the set but also on the family $\textbf{F}$.
   \begin{example}
     Consider set $V=\{y^3=1-x^2\} \subset \mathbb{R}^2$. The compact set $E=\{(x,y) \in V : x \in [-\frac{1}{2},-\frac{1}{4}] \cup [\frac{1}{4},\frac{1}{2}]\}$ is a $\mathcal{P}(y) \otimes \mathcal{P}_{1}(x)$-Markov, but it is not $\mathcal{P}(x) \otimes \mathcal{P}_{2}(y)$-Markov.
   \end{example}
\begin{proof}
  The fact that $E=\{(x,y) \in V : x \in [-\frac{1}{2},-\frac{1}{4}] \cup [\frac{1}{4},\frac{1}{2}]\}$ is a $\mathcal{P}(y) \otimes \mathcal{P}_{1}(x)$-Markov follows form \cite{BK} and \cite{Bor}. So we need only show that $E$ is not $\mathcal{P}(x) \otimes \mathcal{P}_{2}(y)$-Markov. Seeking a contradiction, we consider the sequence of polynomials
       \begin{align*}
         P_n(x,y)=y-\sum_{k=0}^{n} \frac{\Gamma(k-1/3)}{\Gamma(-1/3) k! } x^{2k}.
       \end{align*}
It is well known that
      \begin{align*}
        \sqrt[3]{1-x^2}=\sum_{k=0}^{\infty} \frac{\Gamma(k-1/3)}{\Gamma(-1/3) k! } x^{2k} \quad \text{for} \quad |x| \leq 1.
      \end{align*}
Hence
     \begin{align*}
       \|P_n(x,y)\|_E&=\left\|\sum_{k=n+1}^{\infty} \frac{\Gamma(k-1/3)}{\Gamma(-1/3) k! } x^{2k}\right\|_{[-\frac{1}{2},-\frac{1}{4}] \cup [\frac{1}{4},\frac{1}{2}]}\\&=\left\| \frac{x^{2+2 k} \Gamma\left(\frac{1}{3} (2+3 k)\right) F\left(1,\frac{2}{3}+k,2+k,x^2\right)}{\Gamma\left(-\frac{1}{3}\right) \Gamma(2+k)} \right\|_{[-\frac{1}{2},-\frac{1}{4}] \cup [\frac{1}{4},\frac{1}{2}]}
     \end{align*}
where $F$ is the hypergeometric function defined for $|z| < 1$ by the power series
     \begin{align*}
       F(a,b;x;z)=\sum_{n=0}^{\infty} \frac{(a)_n (b)_n}{(c)_n} \frac{z^n}{n!}.
     \end{align*}
Here $(q)_n$ is the (rising) Pochhammer symbol.
If $x \in [0,1]$, then $F\left(1,\frac{2}{3}+k,2+k,x^2\right)$ is the increasing function of $x$, since its Taylor coefficients are all positive. Therefore, by $F(a,b;c;1)=\frac{\Gamma(c) \Gamma(c-a-b)}{\Gamma(c-a)\Gamma(c-b)}$ and $z\Gamma(z)=\Gamma(z+1)$, we have
     \begin{align*}
       F\left(1,\frac{2}{3}+k,2+k,x^2\right) \leq F\left(1,\frac{2}{3}+k,2+k,1\right)=\frac{ \Gamma(2+k) \Gamma\left(\frac{1}{3}\right) }{\Gamma\left(\frac{4}{3}\right) \Gamma(k+1)}=3(1+k)
     \end{align*}
If we recall that $\lim_{n\rightarrow \infty} \frac{\Gamma(n+\alpha)}{\Gamma(n)n^{\alpha}}=1$, then we may conclude that
     \begin{align*}
       \lim_{n\rightarrow \infty} n^r \|P_n(x,y)\|_E=0 \quad \text{for all} \quad r>0.
     \end{align*}
This gives a contradiction, and the result is established.
\end{proof}
The Example \ref{Ex1} illustrates the more general idea.
\begin{example}
Combining methods used in \cite{BK} with method from Example \ref{Ex1}, one can provide other examples of $\mathcal{P}(y) \otimes \mathcal{P}_{2}(x_k)$-Markov sets by considering algebraic sets of the form
   \begin{align*}
     V=\{ (x_1,\ldots,x_N) \in \mathbb{R}^N : x_k^3=Q(y)x_k\}
   \end{align*}
where $Q_j \in \mathcal{P}(y)$ and $y=(x_1,\ldots,x_{k-1},x_{k+1},\ldots,x_N) \in \mathbb{R}^{N-1}$.
\end{example}
\section{$C^{\infty}$ functions}
First we introduce the subspace of the space $C^{\infty}(\mathbb{R}^N)$ related to an algebraic set defined by (\ref{V}). We define
    \begin{align}\label{C}
      C^{\infty}_V(\mathbb{R}^N):=\left\{f \in C(\mathbb{R}^N) :  \forall_{r>0} \lim_{n\rightarrow \infty} n^r\text{dist}_I\left(f,\mathcal{P}_n(y) \otimes \mathcal{P}_{d-1}(x_k)\right) \right. \nonumber \\ \left. \text{ for every compact cube } I \text{ in } \mathbb{R}^N\right\}.
    \end{align}
Since every cube $I$ is a Markov set, then by Ple\'{s}niak's theorem (see \cite{P1}) $C^{\infty}_V(\mathbb{R}^N) \subset C^{\infty}(\mathbb{R}^N)$. It should be noted that Ple\'{s}niak's result, together with the Jackson theorem, implies
\begin{align*}
      C^{\infty}(\mathbb{R}^N)=\left\{f \in C(\mathbb{R}^N) :  \forall_{r>0} \lim_{n\rightarrow \infty} n^r\text{dist}_I\left(f,\mathcal{P}_n(x_1,\ldots,x_N)\right) \right. \\ \left. \text{ for every compact cube } I \text{ in } \mathbb{R}^N\right\}.
    \end{align*}
We say that $f$ is a $C^{\infty}_V$ function on a compact subset $E$ of $V$ if, there exists a function $\tilde{f} \in C^{\infty}_V(\mathbb{R}^N)$ with $\tilde{f}_{|E}=f$. We denote by $C^{\infty}_V(E)$  the space of such functions. Following Zerner \cite{Z}, similarly as in \cite{P1},  we introduce in $C^{\infty}_V(E)$ the seminorms $\delta_{-1}(f):=\|f\|_E$, $\delta_0(f):=\text{dist}_E(f,\mathcal{P}_0(y) \otimes \mathcal{P}_{d-1}(x_k))$ and
     \begin{align*}
       \delta_\nu(f):=\sup_{l \geq 1} l^\nu \text{dist}_E(f,\mathcal{P}_l(y) \otimes \mathcal{P}_{d-1}(x_k))
     \end{align*}
for $\nu=1,2,\ldots$. By the definition of the set $C^{\infty}_V(\mathbb{R}^N)$ the $\delta_\nu$'s are indeed seminorms on $C^{\infty}_V(E)$. Let $\tau_J$ be the topology for $C^{\infty}_V(E)$ determined by the seminorms $\delta_\nu \, (\nu =-1,0, \ldots)$. In general, this topology is not complete. \newline
The natural topology $\tau_0$ on the set $C^{\infty}(\mathbb{R}^N)$ is  determined by the seminorms $|\cdot|^\nu_K$,
       where for each compact set $K$ in $\mathbb{R}^N$ and each $\nu=0,1,\ldots$,
    \begin{align*}
      |f|^\nu_K:=\max_{|\alpha| \leq \nu} \|D^{\alpha} f\|_K.
    \end{align*}
Therefore we consider the topology $\tau_Q$ for $C^{\infty}_V(E)$ determined by the seminorms
    \begin{align*}
      q_{K,\nu}(f):=\inf\left\{|\tilde{f}|^\nu_K : f \in  C^{\infty}_V(\mathbb{R}^N), \,\tilde{f}_{|E}=f \right\}
    \end{align*}
Then $\tau_Q$ is exactly the quotient topology of the space $C^{\infty}_V(\mathbb{R}^N)/I(E)$, where $C^{\infty}_V(\mathbb{R}^N)$ is endowed with the natural topology $\tau_0$ and $I(E):=\{f \in C^{\infty}_V(\mathbb{R}^N) : f_{|E}=0\}$. Since $( C^{\infty}(\mathbb{R}^N), \tau_0)$ is
complete and $C^{\infty}_V(\mathbb{R}^N)$ is a closed subspace of $C^{\infty}(\mathbb{R}^N)$, the space $( C^{\infty}_V(\mathbb{R}^N), \tau_0)$ is also complete. In view of the fact that $I(E)$ is a closed subspace of $( C^{\infty}_V(\mathbb{R}^N), \tau_0)$,  the quotient space $C^{\infty}_V(\mathbb{R}^N)/I(E)$ is complete, whence $(C^{\infty}_V(E), \tau_Q)$ is a Fr\'{e}chet space.

To prove the main result, we will need the following lemma (see, e.g., \cite{M}, 1.4.2).
   \begin{lem}\label{fh}
     There are positive constants $C_{\alpha}$ depending only on $\alpha \in \mathbb{Z}^N_{+}$ such that for each compact set $K$ in $\mathbb{R}^N$ and each $\epsilon > 0$, one can find a $C^{\infty}$ function $h$ on $\mathbb{R}^N$ satisfying $0\leq h \leq 1$ on $\mathbb{R}^N$, $h=1$ in a neighborhood of $K$, $h(x)=0$ if $\text{dist}(x,K) > \epsilon$, and for all $x \in \mathbb{R}^N$ and $\alpha \in \mathbb{Z}^N_{+}$, $|D^{\alpha} h(x)| \leq C_{\alpha} \epsilon^{-|\alpha|}$.
   \end{lem}
\section{Main result}
Before starting the main result, we prove the following lemma.
   \begin{lem}\label{Gi}
     Let $E$ be a $\mathcal{P}(y) \otimes \mathcal{P}_{d-1}(x_k)$-Markov set. Also define
    \begin{align*}
      \pi(E)=\left\{y=(x_1,\ldots,x_{k-1},x_{k+1},\ldots,x_N) \in \mathbb{R}^{N-1}  : (x_1,\ldots,x_N) \in E, \,  x_k \in \mathbb{R}\right\}.
    \end{align*}
     If $E$ is a $\mathcal{P}(y) \otimes \mathcal{P}_{d-1}(x_k)$-Markov set (with $M$ and $m$), then $\pi(E)$ is a Markov set (as a subset of $\mathbb{R}^{N-1}$) and for every polynomial $P=\sum_{i=0}^{d-1} G_i(y)x_k^i$ there exist constant $C>0$ (depending only on $E$ and $d$) such that
        \begin{align*}
         \|G_i\|_{\pi(E)} \leq \frac{C}{i!} (\deg P)^{m(d-1)} \|P\|_E,
        \end{align*}
   for every $i=0,1,\ldots,d-1$. Conversely, if $\pi(E)$ is a Markov set (with $A$ and $\eta$) and for every polynomial $P=\sum_{i=0}^{d-1} G_i(y)x_k^i$ there exist  $B, \lambda >0$ (depending only on $E$ and $d$) such that
        \begin{align}\label{bounds}
         \|G_i\|_{\pi(E)} \leq B (\deg P)^{\lambda} \|P\|_E, \quad i=0,1,\ldots,d-1,
        \end{align}
   then $E$ is a $\mathcal{P}(y) \otimes \mathcal{P}_{d-1}(x_k)$-Markov set.
   \end{lem}
      \begin{proof}
       Let $E$ be a $\mathcal{P}(y) \otimes \mathcal{P}_{d-1}(x_k)$-Markov set. The proof starts from the observation that
        \begin{align*}
        \frac{\partial^{d-1} P}{\partial x_k^{d-1}} = (d-1)!G_{d-1}
        \end{align*}
      Therefore the $\mathcal{P}(y) \otimes \mathcal{P}_{d-1}(x_k)$-Markov property of the set $E$ gives
        \begin{align*}
          \|G_{d-1}\|_{\pi(E)} \leq \frac{M^{d-1}}{(d-1)!} (\deg P)^{m(d-1)} \|P\|_E.
        \end{align*}
      If $i=d-2$ then
       \begin{align*}
         (d-2)! G_{d-2} =\frac{\partial^{d-2} P}{\partial x_k^{d-2}} - (d-1)G_{d-1}x_k
       \end{align*}
      Hence there exists constant $C>0$ (depending only on the set $E$) such that
       \begin{align*}
          \|G_{d-2}\|_{\pi(E)} \leq \frac{(C+1)M^{d-1}}{(d-2)!} (\deg P)^{m(d-1)} \|P\|_E.
       \end{align*}
      Continuing this process one can show that there exist constant $C_1>0$ (depending only on the set $E$ and $d$) such that
       \begin{align*}
         \|G_{i}\|_{\pi(E)} \leq   \frac{C_1}{i!} (\deg P)^{m(d-1)} \|P\|_E.
       \end{align*}
       To prove the converse direction, assume that $ \pi(E)$ is a Markov set and (\ref{bounds}) holds. Then for every polynomial $P=\sum_{i=0}^{d-1} G_i(y)x_k^i$ we have
          \begin{align*}
            \left\|\frac{\partial P}{\partial x_j} \right\|_E \leq \sum_{i=0}^{d-1} \left\| \frac{\partial G_i }{\partial x_j} x_k^i +  G_i \frac{\partial x_k^i}{\partial x_j} \right\|_E.
          \end{align*}
      Since $E$ is compact, there exists $K>0$, depending only on the set $E$, such that
          \begin{align*}
            \left\| \frac{\partial G_i }{\partial x_j} x_k^i +  G_i \frac{\partial x_k^i}{\partial x_j} \right\|_E \leq K \left( \left\| \frac{\partial G_i }{\partial x_j}\right\|_{\pi(E)} + \|G_i\|_{\pi(E)} \right),
          \end{align*}
      for every $ j=1,2,\ldots,N$ and  $i=0,1,\ldots,d-1$.
      Therefore
          \begin{align*}
             \left\|\frac{\partial P}{\partial x_j} \right\|_E \leq K \sum_{i=0}^{d-1}  \left\| \frac{\partial G_i }{\partial x_j}\right\|_{\pi(E)} + \|G_i\|_{\pi(E)}.
          \end{align*}
      Then, using the fact that $ \pi(E)$ is a Markov set, there exists constants $A >0$ and $\eta>0$ such
      that
          \begin{align*}
            \left\|\frac{\partial P}{\partial x_j} \right\|_E \leq K \sum_{i=0}^{d-1} A (\deg G_i)^{\eta} \|G_i\|_{\pi(E)} + \|G_i\|_{\pi(E)}.
          \end{align*}
      Finally, we use (\ref{bounds}) to see that
          \begin{align*}
            \left\|\frac{\partial P}{\partial x_j} \right\|_E \leq Kd \left( AB (\deg P)^{\eta + \lambda}  + B (\deg P)^{\lambda} \right)\|P\|_{E}.
          \end{align*}
      That concludes the proof.
      \end{proof}
We say that the set $E \subset V$ is $C^{\infty}_V$ determining if for each $f \in C^{\infty}_V(\mathbb{R}^N)$, $f_{|E}=0$ implies $D^{\alpha}f_{|E}=0$, for all $\alpha \in \mathbb{Z}^N_{+}$.
Now, our main result reads as follows.
   \begin{thm}
     If E is a $C^{\infty}_V$ determining compact subset of $V$ then the following statements are equivalent:
     \begin{description}
       \item[(i)] ($\mathcal{P}(y) \otimes \mathcal{P}_{d-1}(x_k)$-Markov Inequality) There exist positive constants $M$ and $r$ such that for each polynomial $P \in \mathcal{P}(y) \otimes \mathcal{P}_{d-1}(x_k)$ and each $\alpha \in \mathbb{Z}^N_{+}$,
             \begin{align*}
               \|D^{\alpha} P\|_E \leq M (\deg P)^{r|\alpha|} \|P\|_E.
             \end{align*}
       \item[(ii)] There exist positive constants $M$ and $r$ such that for every \\ $P \in \mathcal{P}(y) \otimes \mathcal{P}_{d-1}(x_k)$ of degree at most $n$, $n = 1,2, \ldots$,
              \begin{align*}
                |P(x)| \leq M \|P\|_E  \quad \text{if} \quad x \in E_n:=\{x \in \mathbb{R}^N : \text{dist}(x,E) \leq 1/n^r\}.
              \end{align*}
       \item[(iii)] (Bernstein's Theorem) For every function $f: E \rightarrow \mathbb{R}$, if the sequence \\$\{\text{dist}_E(f,\mathcal{P}_l(y) \otimes \mathcal{P}_{d-1}(x_k))\}$  is rapidly decreasing then there is a $C^{\infty}_V(\mathbb{R}^N)$ function $\tilde{f}$ on $\mathbb{R}^N$ such that $\tilde{f}_{|E}=f$.
       \item[(iv)] The space $(C^{\infty}_V(E), \tau_J)$ is complete and $C^{\infty}_V(E)=C^{\infty}(E)$.
       \item[(v)]  The topologies $\tau_J$ and $\tau_Q$ for $C^{\infty}_V(E)$ coincide.
     \end{description}
   \end{thm}
\begin{proof}
The proof of equivalence of (\textbf{i}) and (\textbf{ii}) is almost the same as in \cite{P1}, and we omit the details.  Next we show that $(\textbf{i}) \Leftrightarrow (\textbf{iii}) \Leftrightarrow (\textbf{iv}) \Leftrightarrow (\textbf{v})$. Suppose that we have function
$f: E \rightarrow \mathbb{R}$ such that for each $s > 0$,
   \begin{align*}
     \lim_{l \rightarrow \infty} l^s \|f - P_l\|_E =0.
   \end{align*}
Here $P_l=\sum_{i=0}^{d-1} G_{l,i}(y)x_k^i$ is a metric projection of $f$ onto $\mathcal{P}_l(y) \otimes \mathcal{P}_{d-1}(x_k)$ ($l=0,1,\ldots$). Set, as in Lemma \ref{Gi},
    \begin{align*}
      \pi(E)=\{y=(x_1,\ldots,x_{k-1},x_{k+1},\ldots,x_N) \in \mathbb{R}^{N-1}  : (x_1,\ldots,x_N) \in E, \, x_k \in \mathbb{R}\}.
    \end{align*}
We assume that $r$ is an integer so large that both (\textbf{i}) and (\textbf{ii}) are valid for $E$. Let $\epsilon_l=1/l^r$  and for $l= 1, 2,\ldots$ take a function $h_l \in C^{\infty}(\mathbb{R}^{N-1})$ of Lemma \ref{fh} corresponding to $\epsilon_l$ and $\pi(E)$. We will show that
    \begin{align*}
      \tilde{f}(x_1,\ldots,x_N):=\sum_{i=0}^{d-1} G_{0,i}(y)x_k^i + \sum_{l=1}^{\infty}  \sum_{i=0}^{d-1} h_l(y) (G_{l,i}(y)-G_{l-1,i}(y))x_k^i
    \end{align*}
determines a function from $C^{\infty}_V(\mathbb{R}^N)$ such that $\tilde{f}_{|E}=f$. In order to prove that $\tilde{f} \in C^{\infty}_V(\mathbb{R}^N)$, it suffices to check that
    \begin{align*}
      G_{0,i}(y) + \sum_{l=1}^{\infty}  h_l(y) (G_{l,i}(y)-G_{l-1,i}(y)) \in C^{\infty}(\mathbb{R}^{N-1}),
    \end{align*}
for every $i=0,1,\ldots,d-1$. Thus, if $\gamma \in \mathbb{Z}^{N-1}_{+}$, then, by (\textbf{i}) and (\textbf{ii}),
    \begin{align*}
      \sup_{\mathbb{R}^{N-1}} |D^{\gamma}(h_l(G_{l,i}-G_{l-1,i}))| \leq& \sum_{\beta \leq \gamma} {\gamma \choose \beta} \sup_{\pi(E)_l} |D^{\beta}h_l D^{\gamma - \beta}(G_{l,i}-G_{l-1,i})| \\
      \leq& M \sum_{\beta \leq \gamma} {\gamma \choose \beta} C_{\beta} l^{r|\beta|} \|D^{\gamma - \beta}(G_{l,i}-G_{l-1,i})\|_{\pi(E)} \\ \leq& M_1 l^{r|\gamma|} \|G_{l,i}-G_{l-1,i}\|_{\pi(E)}.
    \end{align*}
where $\pi(E)_l:=\{y \in \mathbb{R}^{N-1} : \text{dist}(y,\pi(E)) \leq \epsilon_l \}$.
From Lemma \ref{Gi} there is a constant $C>0$ so that
     \begin{align*}
        \sup_{\mathbb{R}^{N-1}} |D^{\gamma}(h_l(G_{l,i}-G_{l-1,i}))| \leq C (\deg P_l)^{r(|\gamma|+d-1)} \|P_l - P_{l-1}\|_E \leq
     \end{align*}
Now if $l \geq \max\{2,d\}$, then
      \begin{align*}
        \sup_{\mathbb{R}^{N-1}} |D^{\gamma}(h_l(G_{l,i}-G_{l-1,i}))| \leq 2C l^{-2} \delta_{2r(|\gamma|+d-1)+2}(f)
      \end{align*}
with a constant $C$ independent of $l$. \newline
\indent  Assume now (iii). It is clear that $C^{\infty}_V(E)=C^{\infty}(E)$ and (iv) follows from continuity of the map $C(E) \ni f \rightarrow \text{dist}_E(f,\mathcal{P}_l(y) \otimes \mathcal{P}_{d-1}(x_k)) \in \mathbb{R}$. Let now $I$ be a
compact cube in $\mathbb{R}^{N}$ containing $E$ in its interior. By Jackson's theorem, for every $\nu$ there is a constant $C_\nu >0$ such that for each $f \in C^{\infty}_V(\mathbb{R}^N)$,
      \begin{align*}
         \delta_\nu(f) \leq C_\nu q_{I,\nu+1} (f).
      \end{align*}
Hence, if $(C^{\infty}_V(E), \tau_J)$ is complete, by Banach's theorem the topologies $\tau_J$ and $\tau_Q$ are equal, and we get (\textbf{v}). \newline
The final step of the proof is to show that (\textbf{v}) implies  (\textbf{i}). If topologies $\tau_J$ and $\tau_Q$ coincide, there are a positive constant $M$ and an integer $\mu \geq - 1$ such that for each $f \in C^{\infty}_V(E)$, we have $q_{E,1}(f) \leq M \delta_{\mu}(f)$. Since  $\pi(E)$ is $C^{\infty}$ determining and $\delta_{0}(f) \leq \|f\|$, we must have $\mu \geq 1$.  In particular, if $f \in \mathcal{P}_\lambda(y) \otimes \mathcal{P}_{d-1}(x_k)$ , we get
        \begin{align*}
          \left\| \frac{\partial f}{\partial x_j} \right\|_E \leq M \sup_{1 \leq l \leq \lambda} l^{\mu} \text{dist}_E(f,\mathcal{P}_l(y) \otimes \mathcal{P}_{d-1}(x_k)) \leq M \lambda^{\mu} \|f\|_E
        \end{align*}
for $j= 1, 2,\ldots, N$, which implies that $E$ is a $\mathcal{P}(y) \otimes \mathcal{P}_{d-1}(x_k)$-Markov set. (We needed the assumption that $E$ is $C^{\infty}_V$  determining.)
\end{proof}
  \begin{rem}
  It is worth noting that, if $E$ satisfies (\textbf{i}) then $E$ must be $C^{\infty}_V$  determining. To see this, take a compact cube $I$ in $\mathbb{R}^N$ containing $E$ in its interior. We let $f \in C^{\infty}_V(\mathbb{R}^N)$ such that $f=0$ on $E$. It follows from the definition of $C^{\infty}_V(\mathbb{R}^N)$ that
     \begin{align*}
        \epsilon_l:=\text{dist}_I (f,\mathcal{P}_l(y) \otimes \mathcal{P}_{d-1}(x_k))=\|f-P_l\|_I=\|f-\sum_{i=0}^{d-1} G_{l,i}(y)x_k^i\|_I
     \end{align*}
  is rapidly decreasing and by Markov's inequality, for each $\alpha \in \mathbb{Z}^N_{+}$, we have
      \begin{align*}
        D^{\alpha}f= \sum_{i=0}^{d-1} D^{\alpha}\{ G_{0,i}(y)x_k^i\} + \sum_{l=1}^{\infty}  \sum_{i=0}^{d-1}  D^{\alpha} \{(G_{l,i}(y)-G_{l-1,i}(y))x_k^i\} \quad \text{on } I.
      \end{align*}
  Finally, by $(\textbf{i})$, we obtain that $D^{\alpha}f(x)=\lim_{l\rightarrow \infty} D^{\alpha}P_l(x)=0$ for every $x \in E$.

  \end{rem}
\section*{Acknowledgment}
The author was supported by the Polish National Science Centre (NCN) Opus grant no. 2017/25/B/ST1/00906.
\bibliographystyle{amsplain}

\noindent Tomasz Beberok\\
Department of Applied Mathematics,\\
University of Agriculture in Krakow,\\
ul. Balicka 253c, 30-198 Kraków, Poland\\
email: tbeberok@ar.krakow.pl

\end{document}